\documentclass[graybox]{svmult}

\usepackage{mathptmx}       
\usepackage{helvet}         
\usepackage{courier}        
\usepackage{type1cm}        
\usepackage{makeidx}         
\usepackage{graphicx}        
\usepackage{multicol}        
\usepackage[bottom]{footmisc}

\makeindex             

\usepackage{amsmath,amssymb}

\newcommand{\wtisign}{\operatorname{sign}}
\newcommand{\wtivar}{\operatornamewithlimits{\bf var}}
\newcommand{\wtiR}{{\mathbb{R}}}
\newcommand{\wtiF}{{\mathcal{F}}}
\newcommand{\wtiE}{ {\bf E}}
\newcommand{\wtiProb}{ {\bf P}}
\newcommand{\wtinorm}[1]{\left\|#1\right\|}
\newcommand{\wtiabs}[1]{\left|#1\right|}
\newcommand{\wtiindic}{\mathbf{1}}
\newcommand{\wtipr}[1]{\mathbf{P}\left(#1\right)}

\begin{document}

\title*{Replication of Wiener-transformable stochastic processes  with application to financial markets with  memory}

\titlerunning{Replication of Wiener-transformable processes}

\author{Elena Boguslavskaya, Yuliya Mishura, and Georgiy Shevchenko}
\institute{
Elena Boguslavskaya \at
Department of  Mathematics, Brunel University London, Uxbridge UB8 3PH, UK\\
\email{elena@boguslavsky.net}
\and
Yuliya Mishura \at
Department of Probability Theory, Statistics and Actuarial Mathematics,
Taras Shevchenko National University of Kyiv,
64, Volodymyrs'ka St.,
 01601 Kyiv, Ukraine\\
\email{myus@univ.kiev.ua}
\and
Georgiy Shevchenko\at
Department of Probability Theory, Statistics and Actuarial Mathematics,
Taras Shevchenko National University of Kyiv,
64, Volodymyrs'ka St.,
 01601 Kyiv, Ukraine\\
\email{zhora@univ.kiev.ua}}
%
%
\maketitle

\abstract*{We investigate Wiener-transformable markets, where the driving process is given by an adapted transformation of a Wiener process. This includes processes with long memory, like fractional Brownian motion and related processes, and, in general,
Gaussian processes   satisfying certain regularity conditions on their covariance functions. Our choice of  markets is motivated by  the well-known phenomena of  the so-called ``constant'' and ``variable depth''  memory observed in real world price processes,  for which fractional and multifractional models are the most adequate descriptions. Motivated by integral representation results in general Gaussian setting, we study the conditions under which random variables can be represented as pathwise integrals with respect to the driving process. From financial point of view, it means that we give  the conditions of replication of contingent claims on such markets. As an application of our results, we consider the utility maximization problem in our specific setting. Note that the markets under consideration can be both arbitrage  and arbitrage-free, and moreover, we give the representation results in terms of bounded strategies.
}

\abstract{We investigate Wiener-transformable markets, where the driving process is given by an adapted transformation of a Wiener process. This includes processes with long memory, like fractional Brownian motion and related processes, and, in general,
Gaussian processes   satisfying certain regularity conditions on their covariance functions. Our choice of  markets is motivated by  the well-known phenomena of  the so-called ``constant'' and ``variable depth''  memory observed in real world price processes,  for which fractional and multifractional models are the most adequate descriptions. Motivated by integral representation results in general Gaussian setting, we study the conditions under which random variables can be represented as pathwise integrals with respect to the driving process. From financial point of view, it means that we give  the conditions of replication of contingent claims on such markets. As an application of our results, we consider the utility maximization problem in our specific setting. Note that the markets under consideration can be both arbitrage  and arbitrage-free, and moreover, we give the representation results in terms of bounded strategies.
}


\section{Introduction}

Consider a general continuous time market model with one risky asset. For simplicity, we will work with discounted values. Let the stochastic process $\{X_t,t\in [0,T]\}$ model the discounted price of risky asset. Then the discounted final value of a self-financing portfolio\index{portfolio}  is given by a stochastic integral
\begin{equation}\label{wti:capital}
V^\psi(T) = V^\psi(0) + \int_0^T \psi(t) dX(t),
\end{equation}
where an adapted process $\psi$ is the quantity of risky asset in the portfolio. Loosely speaking, the self-financing assumption means that no capital is withdrawn or added to the portfolio; for precise definition and general overview of financial market models with continuous time we refer a reader to \cite{bjork,karat}.

Formula \eqref{wti:capital} raises several important questions of financial modeling, we will focus here on the following two.
\begin{itemize}
\item \textit{Replication}: \index{replication} identifying random variables (i.e.\ discounted contingent claims), which can be represented as final capitals of some self-financing portfolios. In other words, one looks at integral representations
\begin{equation}\label{wti:represent}
\xi = \int_0^T \psi(t) dX(t)
\end{equation}
with adapted integrand $\psi$; the initial value may be subtracted from $\xi$, so we can assume that it is zero.
\item \textit{Utility maximization}: \index{utility maximization} maximizing the expected utility of final capital over some set of admissible self-financing portfolios.
\end{itemize}

An important issue is the meaning of stochastic integral in \eqref{wti:capital} or \eqref{wti:represent}. When the process $X$ is a semimartingale, it can be understood as It\^o integral. In this case \eqref{wti:capital} is a kind of It\^o representation, see e.g.\ \cite{KS} for an extensive coverage of this topic. When the It\^o integral is understood in some extended sense, then the integral representation may exist under very mild assumptions and may be non-unique. For example, if $X=W$, a Wiener process, and $\psi$ satisfies $\int_0^T \psi_s^2 ds<\infty$ a.s., then, as it was shown by \cite{dudley}, any random variable can be represented as a final value of some self-financing portfolio for any value of initial capital.

However, empirical studies suggest that financial markets often exhibit long-range dependence (in contrast to stochastic volatility that can be both smooth and rough, i.e., can demonstrate both long-and  short-range dependence). The standard model for the phenomenon of   long-range dependence is the fractional Brownian motion with Hurst index $H>1/2$. It is not a semimartingale, so the usual It\^o integration theory is not available. The standard approach now is to define the stochastic integral in such models as a pathwise integral, namely, one usually considers the fractional integral, see \cite{bender-sottinen-valkeila,zahle}.

The models based on the fractional Brownian motion usually admit arbitrage possibilities, i.e.\ there  self-financing portfolios $\psi$ such that $V_\psi(0)\le 0$, $V_\psi(T)\ge 0$ almost surely, and $V_\psi(T)>0$ with positive probability. In the fractional Black--Scholes model, where $X_t=X_0\exp\{at+bB_t^H\}$, and $B^H$ is a fractional Brownian motion with $H>1/2$, the existence of arbitrage was shown in \cite{rogers}. Specifically, the strategy constructed there was of a ``doubling'' type, blowing the portfolio in the case of negative values; thus the potential intermediate losses could be arbitrarily large. It is worth to mention that such arbitrage exists even in the classical Black--Scholes model: the aforementioned result by Dudley allows gaining any positive final value of capital from initial zero by using a similar ``doubling'' strategy. For this reason, one usually restricts the class of admissible strategies by imposing a lower bound on the running value:\index{non-doubling strategy}
\begin{equation}\label{wti:eq:nds}
V^\psi(t)\ge -a,\quad t\in(0,T),
\end{equation}
which in particular disallows the ``doubling'' strategies. However, in the fractional Black--Scholes model, the arbitrage exists even in the class of  strategies satisfying \eqref{wti:eq:nds}, as was shown in \cite{Cheridito1}.

There are several ways to exclude arbitrage in the fractional Brownian model\index{fractional Brownian motion}. One possibility is to restrict the class of admissible strategies. For example, in \cite{Cheridito1} the absence of arbitrage is proved under further restriction that interval between subsequent trades is bounded from below (i.e.\ high frequency trading is prohibited). Another possibility is to add to the fractional Brownian motion an independent Wiener process, thus getting the so-called mixed fractional Brownian motion\index{mixed fractional Brownian motion} $M^H = B^H + W$.
The absence in such mixed models was addressed in \cite{andrmish,Cheridito}.  In \cite{andrmish}, it was shown that there is no arbitrage in the class of self-financing strategies $\gamma_t = f(t,M^H,t)$ of Markov type, depending only on the current value of the stock. In \cite{Cheridito}, it was shown that for $H\in(3/4,1)$ the distribution of mixed fractional Brownian motion on a finite interval is equivalent to that of Wiener process. As a result, in such models there is no arbitrage strategies satisfying the non-doubling assumption \eqref{wti:eq:nds}. A more detailed exposition concerning arbitrage in models based on fractional Brownian motion is given in \cite{bender-sottinen-valkeila1}.

The replication question, i.e.\ the question when a random variable can be represented as a pathwise (fractional) integral in the models with long memory was studied in many articles, even in the case where arbitrage opportunities are present. The first results were established in \cite{msv}, where it was shown that a random variable $\xi$ has representation \eqref{wti:represent} with respect to fractional Brownian motion if it is a final value of some H\"older continuous adapted process. The assumption of H\"older continuity might seem too restrictive at the first glance. However, the article \cite{msv} gives numerous examples of random variables satisfying this assumption.

The results of \cite{msv} were extended in \cite{shev-viita}, where similar results were shown for a wide class of Gaussian integrators. The article \cite{mish-shev} extended them even further and studied when a combination of H\"older continuity of integrator and small ball estimates lead to existence of representation \eqref{wti:represent}.

For the mixed fractional Brownian motion, the question of replication was considered in \cite{shev-viita}. The authors defined the integral with respect to fractional Brownian motion in pathwise sense and that with respect to Wiener process in the extended It\^o sense and shown, similarly to the result of \cite{dudley}, that any random variable has representation \eqref{wti:represent}.

It is worth to mention that the representations constructed in \cite{msv,mish-shev,shev-viita} involve integrands of ``doubling'' type, so in particular they do not satisfy the admissibility assumption \eqref{wti:eq:nds}.

Our starting point for this article was to see what contingent claims   are  representable as final values of some H\"older continuous adapted processes. It turned out that the situation is quite transparent whenever the Gaussian integrator generates the same flow of sigma-fields as the Wiener process. As a result, we came up with the concept of  Wiener-transformable financial market, which turned out to be a fruitful idea, as a lot of models of financial markets are Wiener-transformable. We consider many examples of such models in our paper. Moreover, the novelty of the present results is that we prove representation theorems that, in financial interpretation, are equivalent to the possibility of hedging of contingent claims, in the class of \textit{bounded} strategies. While even with such strategies the non-doubling assumption \eqref{wti:eq:nds} may fail, the boundedness  seems a feasible admissibility assumption.

More specifically, in the present  paper  we study  a replication and the utility maximization problems  for a broad class of asset prices processes, which are obtained by certain adapted transformation of a Wiener process; we call such processes \textit{Wiener-transformable} and provide several examples.   We concentrate mainly on non-semimartingale markets because the semimartingale markets have been studied thoroughly in the literature. Moreover, the novelty of the present results is that we prove representation theorems that, in financial interpretation, are equivalent to the possibility of hedging of contingent claims, in the class of bounded strategies. We would like to draw the attention of the reader once again to the fact that the possibility of representation means that we have arbitrage possibility in the considered class of strategies and they may be limited, although in a narrower and more familiar class of   strategies the market can be arbitrage-free. Therefore, our results demonstrate rather subtle differences in the properties of markets in different classes of strategies.

The article is organized as follows. In Section~\ref{wti:sec:2}, we recall basics of pathwise integrations in the fractional sense. In  Section~\ref{wti:sec:3}, we prove a new representation result, establishing an existence of integral representation with bounded integrand, which is of particular importance in financial applications. We also define the main object of study, Wiener-transformable markets, and provide several examples. Section~\ref{wti:sec:4} is devoted to application of representation results to the utility maximization problems.

\section{Elements of fractional calculus}\label{wti:sec:2}
As announced in the introduction, the integral with respect to Wiener-transformable processes will be defined in pathwise sense, as fractional integral.
Here we present the basic facts on fractional integration; for more details see \cite{samko,zahle}. Consider functions  $f,g:[0,T]\rightarrow \mathbb{R}$, and let  $[a,b]\subset [0,T]$.
For $\alpha\in (0,1)$ define Riemann-Liouville fractional derivatives on finite interval $[a,b]$\index{Riemann-Liouville fractional derivative}
\begin{gather*}
\big(\mathcal{D}_{a+}^{\alpha}f\big)(x)=\frac{1}{\Gamma(1-\alpha)}\bigg(\frac{f(x)}{(x-a)^\alpha}+\alpha
\int_{a}^x\frac{f(x)-f(u)}{(x-u)^{1+\alpha}}du\bigg)1_{(a,b)}(x),\end{gather*}
\begin{gather}\label{wti:equ:dif}\big(\mathcal{D}_{b-}^{ \alpha}g\big)(x)=\frac{1} {\Gamma(1-\alpha)}\bigg(\frac{g(x)}{(b-x)^{ \alpha}}+ \alpha
\int_{x}^b\frac{g(x)-g(u)}{(u-x)^{1+\alpha}}du\bigg)1_{(a,b)}(x).
\end{gather}
Assuming that
 $\mathcal{D}_{a+}^{\alpha}f\in L_1[a,b]$, $\mathcal{D}_{b-}^{1-\alpha}g_{b-}\in
L_\infty[a,b]$, where $g_{b-}(x) = g(x) - g(b)$,
the generalized Lebesgue--Stieltjes
integral\index{generalized Lebesgue--Stieltjes
integra}
is defined as
\begin{equation*}\int_a^bf(x)dg(x)= \int_a^b\big(\mathcal{D}_{a+}^{\alpha}f\big)(x)
\big(\mathcal{D}_{b-}^{1-\alpha}g_{b-}\big)(x)dx.
\end{equation*}

Let  function $g$ be $\theta$-H\"{o}lder continuous, $g\in C^\theta[a,b]$ with $\theta\in(\frac12,1)$, i.e.\ 
$$
\sup_{t,s\in[0,T],t\neq s}\frac{\wtiabs{g(t)-g(s)}}{\wtiabs{t-s}^\theta}<\infty.
$$
 In order to integrate w.r.t. function  $g$ and to find an upper bound of the integral, fix some $\alpha \in(1-\theta,1/2)$ and introduce the following norm:
\begin{gather*}
\|f\|_{\alpha,[a,b]} = \int_a^b \left(\frac{|{f(s)}|}{(s-a)^\alpha} + \int_a^s \frac{|{f(s)-f(z)}|}{(s-z)^{1+\alpha}}dz\right)ds.
\end{gather*}
For simplicity we abbreviate $\|\cdot\|_{\alpha,t} = \|\cdot\|_{\alpha,[0,t]}$. Denote $$\Lambda_\alpha(g):= \sup_{0\le s<t\le T} |{\mathcal{D}_{t-}^{1-\alpha}g_{t-}}(s)|.$$ In view of H\"{o}lder continuity, $\Lambda_\alpha(g)<\infty$.

Then for any $t\in(0,T]$ and for any $f$ with $\|f\|_{\alpha,t}<\infty$, the integral $\int_0^t f(s) dg(s)$ is well defined as a generalized Lebesgue--Stieltjes integral, and the following bound is evident:
\begin{gather}\label{wti:equ:ineq}
\Big|{\int_0^t f(s)dg(s)}\Big|\le \Lambda_\alpha(g) \|f\|_{\alpha,t}.
\end{gather}
It is well known that in the case if  $f$ is  $\beta$-H\"{o}lder continuous,  $f\in C^\beta[a,b]$, with $\beta+\theta>1$, the generalized Lebesgue--Stieltjes integral $\int_a^bf(x)dg(x)$ exists, equals to the limit of Riemann sums and admits bound \eqref{wti:equ:ineq} for any $\alpha \in(1-\theta, \beta\wedge 1/2)$.

\section{Representation results for Gaussian and Wiener-transformable processes}\label{wti:sec:3}

Let throughout the paper $(\Omega, \mathcal{F},  \wtiProb)$ be   a complete probability space   supporting all stochastic processes mentioned below. Let also $\mathbb{F} = \{\mathcal F_t,t\in[0,T]\}$ be a filtration satisfying standard assumptions. In what follows, the adaptedness of a process $X = \{X(t),t\in[0,T]\}$ will be understood with respect to $\mathbb{F}$, i.e.\ $X$ will be called adapted if for any $t\in[0,T]$, $X(t)$ is $\mathcal{F}_t$-measurable.

We start with representation results, which supplement those of \cite{mish-shev}.

Consider a continuous centered Gaussian process $G$ with incremental variance of $G$ satisfying the following two-sided power bounds for some $H\in (1/2,1)$.
\begin{itemize}
	\item[$(A)$] There exist $C_1, C_2>0$ such that for any $s,t\in [0,T]$
	\begin{equation}\label{wti:eq:helix}
	C_1\left|t-s\right|^{2H}\le\wtiE\left|G(t)-G(s)\right|^2\le C_2 \left|t-s\right|^{2H}.
	\end{equation}
\end{itemize}
Assume additionally  that the increments of $G$ are positively correlated. More exactly, let the following condition hold
\begin{itemize}
\item[$(B)$] For any $0 \le s_1 \le t_1 \le s_2 \le t_2\le T$ 	 $$\wtiE\left(G({t_1})-G({s_1})\right)\left(G({t_2})-G({s_2})\right)\ge0.$$
\end{itemize}
A process satisfying \eqref{wti:eq:helix} is often referred to as a \textit{quasi-helix}.\index{quasi-helix}

Note that the right inequality in \eqref{wti:eq:helix} implies that
\begin{equation}\label{wti:eq:Gmodcont}
\sup_{t,s\in[0,T]}\frac{|G(t)-G(s)|}{|t-s|^{H}|\log(t-s)|^{1/2}} <\infty
\end{equation}
almost surely (see e.g.\ p.~220 in \cite{lifshits}).

We will need the following small deviation estimate for sum of squares of Gaussian random variables, see e.g.\ \cite{lishao}.

\begin{lemma}
\label{wti:lem:small}
Let $\{\xi_i\}_{i=1,\ldots,n}$ be jointly Gaussian centered random variables. For all $x$ such that $0<x<\sum_{i=1}^n \wtiE \xi^2_i$, it holds
\begin{gather*}
\wtipr{\sum_{i=1}^n\xi_i^2\le x}\leq\exp\left\{-\frac{\left(\sum_{i=1}^n \wtiE {\xi_i^2}-x\right)^2}{\sum_{i,j=1}^n (\wtiE {\xi_i\xi_j})^2}\right\}.
\end{gather*}
\end{lemma}

\begin{theorem}\label{wti:thm:representation}
Let a centered Gaussian process $G$ satisfy $(A)$ and $(B)$ and $\xi$ be a random variable such that there exists an adapted $r$-H\"older continuous process $Z$ with $Z(T) = \xi$. There exists a bounded adapted process $\psi$, such that $\left\|\psi\right\|_{\alpha,T}<\infty$ for some $\alpha\in \left(1-H,1\right)$ and $\xi$ admits the representation
\begin{equation}\label{wti:reprez}
\xi=\int_{0}^{T}\psi(s) dG(s),
\end{equation}  almost surely.
\end{theorem}
\begin{remark}
A similar result was proved in  \cite{mish-shev}, Theorem 4.1, which assumed \eqref{wti:eq:helix} with different exponents in the right-hand side and in the left-hand side of the inequality. Having equal exponents allowed us to establish existence of a \textit{bounded} integrand $\psi$, thus extending previous results.
\end{remark}
\begin{proof}
To construct an integrand, we modify ideas of \cite{mish-shev} and \cite{shalaiko}. Throughout the proof, $C$ will denote a generic constant,  while $C(\omega)$, a random constant; their values may change between lines.

Choose some $\alpha \in \big(1-H ,(r+1-H)\wedge \frac12\big)$.

We start with the construction of $\psi$. First take some  $\theta\in (0,1)$,  put $t_n = T-\theta^{n}$, $n\ge 1$, and let $\Delta_n = t_{n+1}-t_n$.  It is easy to see that
\begin{gather}
T-t_n\le C\Delta_n.
\label{wti:t_n-ineq}
\end{gather}
Denote for brevity $\xi_n = Z(t_n)$. Then by Assumption 1, $\wtiabs{\xi_n -\xi_{n+1}}\le C(\omega) \theta^{rn}$. Therefore, there exists some $N_0 = N_0(\omega)$ such that
\begin{equation}\label{wti:eq:deltaxi}
\wtiabs{\xi_n -\xi_{n+1}}\le n \theta^{rn}
\end{equation}
for all $n\ge N_0(\omega)$.

We construct the integrand $\psi$ inductively between the points $\{t_n,n\ge 1\}$. First let $\psi(t)=0$, $t\in[0,t_1]$. Assuming that we have already constructed $\psi(t)$ on $[0,t_n)$, define $V(t)=\int_0^t \psi(s)dG(s), t\in[0,t_n]$.

Consider some cases.

\underline{Case I.} $V(t_n)\neq \xi_{n-1}$. By Lemma 4.1 in \cite{mish-shev}, there exists an adapted process $\{\phi_n(t),t\in[t_n,t_{n+1}]\}$, bounded on $[t_n,t]$ for any $t\in(t_n,t_{n+1})$ and such that $\int_{t_n}^{t}\phi_n(s) d G(s)\to +\infty$ as $t\to t_{n+1}-$. Define a stopping time
\begin{gather*}
\tau_n=\inf\left\{t\geq t_n: \int_{t_n}^{t}\phi_n(s) dG(s)\geq |\xi_n-V_{t_n}| \right\},
\end{gather*}
and set
\begin{gather*}
\psi(t)=\phi_n(t) \wtisign\big(\xi_n-V(t_n)\big)\wtiindic_{[t_n,\tau_n]}(t), \,t\in[t_n,t_{n+1}).
\end{gather*}
It is obvious that $\int_{t_n}^{t_{n+1}}\psi(s)dG(s)=\xi_n-V(t_n)$ and $V(t_{n+1})=\xi_n$.

\underline{Case II.} $V(t_n)=\xi_{n-1}$. We consider a uniform partition $s_{n,k} = t_n + k\delta_n$, $k=1,\ldots,n$ of $[t_n,t_{n+1}]$ with a mesh $\delta_n=\Delta_n/n$ and an auxiliary function
\begin{gather*}
\phi_n(t)=a_n\sum_{k=0}^{n-1}
\big(G(t)-G(s_{n,k})\big)\wtiindic_{[s_{n,k},s_{n,k+1})}(t),
\end{gather*}
where $a_n = n^{-2}\theta^{(\alpha-H-1)n}$. Since $\phi_n$ is piecewise H\"older continuous of order up to $H$, by the change of variables formula (Theorem 4.3.1 in \cite{zahle})
\begin{equation*}
\int_{t_n}^{t_{n+1}} \phi_n(t) dG(t) = a_n \sum_{k=0}^{n-1} \big(G(s_{n,k+1})-G(s_{n,k})\big)^2.
\end{equation*}

Define a stopping time
\begin{gather*}
\sigma_n=\inf\Big\{t\geq t_n: \int_{t_n}^t \phi_n(s)dG(s)\geq |\xi_n-\xi_{n-1}|\Big\}\wedge t_{n+1},
\end{gather*}
and set
\begin{gather*}
\psi(t)=\wtisign(\xi_n-\xi_{n-1})\phi_n(t)\wtiindic_{[t_n,\sigma_n]}(t),\, t\in[t_n,t_{n+1}).
\end{gather*}
Now we want to ensure that, almost surely, $V(t_{n}) = \xi_{n-1}$ for all $n$ large enough. By construction, Case I is always succeeded by Case II. So we need to ensure that $\sigma_n <t_{n+1}$ for all $n$ large enough, equivalently, that
$$
a_n \sum_{k=0}^{n-1} \big(G(s_{n,k+1})-G(s_{n,k})\big)^2> \wtiabs{\xi_n - \xi_{n-1}}.
$$
Thanks to \eqref{wti:eq:deltaxi}, it is enough to ensure that
$$
 \sum_{k=0}^{n-1} \big(G(s_{n,k+1})-G(s_{n,k})\big)^2 > a_n^{-1}n\theta^{rn} = n^2 \theta^{(r+H +1-\alpha)n}
$$
for all $n$ large enough. Define $\xi_k = G(s_{n,k+1})-G(s_{n,k})$, $k=0,\dots,n-1$. Thanks to our choice of $\alpha$, $r+H+1-\alpha > 2H$, so $n^2 \theta^{(r+H +1-\alpha)n} < C_1 n^{1-2H}\theta^{2Hn}$ for all $n$ large enough. Therefore, in view of \eqref{wti:eq:helix},
\begin{gather*}
\sum_{k=0}^{n-1} \wtiE \xi_k^2  \ge C_1 n \delta_n^{2H} = C_1 n^{1-2H}\theta^{2H n} > n^2 \theta^{(r+H +1-\alpha)n},
\end{gather*}
so we can use Lemma \ref{wti:lem:small}.
Using $(A)$ and $(B)$, estimate
\begin{gather*}
\sum_{i,j=0}^{n-1} \big(\wtiE \xi_i\xi_j \big)^2\le \max_{0\le i,j\le n-1} \wtiE \xi_i\xi_j \sum_{i,j=0}^{n-1} \wtiE \xi_i\xi_j\\
 \le C_1 \delta_n^{2H} \wtiE \Big( \sum_{i=0}^{n-1} \xi_i\Big)^2  = C_1 \delta_n^{2H} \wtiE \big(G(t_{n+1}) - G(t_n) \big)^2\\
 \le C_1^2  \delta_n^{2H}\Delta_n^{2H}\le C_1^2 n^{-2H} \Delta^{4H}  = C_1^2 n^{-2H} \theta^{4H n}.
\end{gather*}
Hence, by Lemma~\ref{wti:lem:small},
\begin{gather*}
\wtipr{ \sum_{k=0}^{n-1} \big(G(s_{n,k+1})-G(s_{n,k})\big)^2 \le n^2 \theta^{(r+H +1-\alpha)n}}\\
\le \exp\left\{ - \frac{\big(C_1 n^{1-2H}\theta^{2H n} - n^2 \theta^{(r+H +1-\alpha)n}\big)^2}{C_1^2 n^{-2H} \theta^{4H n}}\right\}
\le \exp \left\{ - C n^{2 -2H}\right\}.
\end{gather*}
Therefore, by the Borel--Cantelli lemma, almost surely there exists some $N_1(\omega)\ge N_0(\omega)$ such that for all $n\ge N_1(\omega)$
$$
\sum_{k=0}^{n-1} \big(G(s_{n,k+1})-G(s_{n,k})\big)^2 > n^2 \theta^{(r+H +1-\alpha)n},
$$
so, as it was explained above, we have $V(t_n) = \xi_{n-1}$, $n\ge N_1(\omega)$.

Since all functions $\phi_n$ are bounded, we have that $\psi$ is bounded on
$[0,t_N]$ for any $N\ge 1$. Further, thanks to \eqref{wti:eq:Gmodcont}, for $t\in[t_n,t_{n+1}]$ with $n\ge N_1(\omega)$,
\begin{equation}\label{wti:eq:psibound}
\begin{gathered}
\wtiabs{\psi(s)}\le C(\omega) a_n \delta_n^H\wtiabs{\log \delta_n}^{1/2} \le  C(\omega) n^{-2}\theta^{(\alpha-H-1)n} n^{-H} \theta^{Hn} n^{1/2}\\ = C(\omega) n^{\alpha - H - 3/2}\theta^{(\alpha-1)n}.
\end{gathered}
\end{equation}
Therefore, $\psi$ is bounded (moreover, $\psi(t)\to 0$, $t\to T-$).

Further, by construction, $\wtinorm{\psi}_{\alpha,t_N}<\infty$ for any $N\ge 1$. Moreover, $\wtiabs{V(t)-\xi_{N-1}}\le \wtiabs{\xi_N - \xi_{N-1}}$, $t\in[t_{N},t_{N+1}]$. Thus, it remains to to verify that $\wtinorm{\psi}_{\alpha,[t_N,1]}<\infty$ and $\int_{t_N}^1 \psi(s) dG(s)\to 0$, $N\to\infty$, which would follow from $\wtinorm{\psi}_{\alpha,[t_N,1]}\to 0$, $N\to\infty$.

Let $N\ge N_1(\omega)$. Write
\begin{equation*} 
\wtinorm{\psi}_{\alpha,[t_N,T]} = \sum_{n=N}^\infty \int_{t_n}^{t_{n+1}}\left(\frac{|\psi(s)|}{(s-t_N)^{\alpha}}+\int_{t_N}^s \frac{|\psi(s)-\psi(u)|}{|s-u|^{1+\alpha}}du\right) ds.
\end{equation*}
Thanks to \eqref{wti:eq:psibound},
\begin{gather*}
\int_{t_n}^{t_{n+1}}\frac{|\psi(s)|}{(s-t_N)^{\alpha}} ds
\le
C(\omega)\Delta_n^{1-\alpha}n^{\alpha - H - 3/2}\theta^{(\alpha-1)n}
= C(\omega)n^{\alpha - H - 3/2}.
\end{gather*}
Further,
\begin{gather*}
\int_{t_N}^{t_{n+1}}\int_{t_n}^s\frac{|\psi(s)-\psi(u)|}{|s-u|^{1+\alpha}}du\, ds\\ =
\sum_{k=1}^n \int_{s_{n,k-1}}^{s_{n,k}}\left(\int_{t_N}^{t_n}+\int_{t_n}^{s_{n,k-1}}+\int_{s_{n,k-1}}^{s} \right)\frac{|\psi(s)-\psi(u)|}{|s-u|^{1+\alpha}}du\, ds=:I_1+I_2+I_3.
\end{gather*}
Start with $I_1$, observing that  $\psi$ vanishes on $(\sigma_n,t_{n+1}]$:
\begin{gather*}
I_1\leq \int_{t_n}^{t_{n+1}}\sum_{j=N}^n \int_{t_{j-1}}^{t_j}\frac{|\psi(s)|+|\psi(u)|}{|s-u|^{1+\alpha}}du\, ds\\
\leq  C(\omega) n^{\alpha - H - 3/2}\theta^{(\alpha-1)n}
\int_{t_n}^{t_{n+1}}(s-t_n)^{-\alpha}ds\\
+ C(\omega)\sum_{j=N}^{n-1} j^{\alpha-H-3/2}\theta^{(\alpha-1)j} \int_{t_n}^{t_{n+1}}(s-t_{j+1})^{-\alpha}ds\\
\leq  C(\omega)n^{\alpha - H - 3/2}\theta^{(\alpha-1)n}\Delta_n^{1-\alpha} + C(\omega) \sum_{j=N}^{n-1}j^{\alpha-H-3/2}\theta^{(\alpha-1)j} \Delta_n^{1-\alpha}\\
= C(\omega)n^{\alpha - H - 3/2} + C(\omega) \sum_{j=N}^{n-1}j^{\alpha-H-3/2}\theta^{(\alpha-1)(j-n)}.
\end{gather*}
Similarly,
\begin{gather*}
I_2\leq C(\omega)n^{\alpha - H - 3/2}\theta^{(\alpha-1)n}  \sum_{k=1}^n \int_{s_{n,k-1}}^{s_{n,k}}\int_{t_n}^{s_{n,k-1}}|s-u|^{-1-\alpha}du\, ds \\
\le C(\omega)n^{\alpha - H - 3/2}\theta^{(\alpha-1)n} \sum_{k=1}^n \int_{s_{n,k-1}}^{s_{n,k}} (s-s_{n,k-1})^{-\alpha}ds\\
\le C(\omega)n^{\alpha - H - 3/2}\theta^{(\alpha-1)n}  n \delta_n^{1-\alpha}=C(\omega)n^{2\alpha - H - 3/2}.
\end{gather*}
Finally, assuming that $\sigma_n\in [s_{n,l-1},s_{n,l})$,
\begin{gather*}
I_3\leq C(\omega)\sum_{k=1}^{l-1} \int_{s_{n,k-1}}^{s_{n,k}}\int_{s_{n,k-1}}^s  a_n\frac{(s-u)^{H}|\log (s-u)|^{1/2}}{(s-u)^{1+\alpha}}du\, ds\\
+ \int_{s_{n,l-1}}^{\sigma_n}\int_{s_{n,l-1}}^s\frac{|\psi(s)-\psi(u)|}{|s-u|^{1+\alpha}}du\, ds+\int_{\sigma_n}^{s_{n,l}}\int_{s_{n,l-1}}^{\sigma_n}\frac{|\psi(s)-\psi(u)|}{|s-u|^{1+\alpha}}du\, ds\\
\le C(\omega)a_n\sum_{k=1}^n \int_{s_{n,k-1}}^{s_{n,k}}(s-s_{n,k-1})^{H-\alpha}|\log(s-s_{n,k-1})|^{1/2}ds\\
+ C(\omega) n^{\alpha - H - 3/2}\theta^{(\alpha-1)n}\int_{\sigma_n}^{s_{n,l}}\int_{s_{n,l-1}}^{\sigma_n}\frac{1}{|s-u|^{1+\alpha}}du\, ds\\ \leq
C(\omega)a_n n\delta_n^{H+1-\alpha}|\log \delta_n|^{1/2}  + C(\omega) n^{\alpha - H - 3/2}\theta^{(\alpha-1)n}\delta_n^{-\alpha}\\
= C(\omega)n^{\alpha-H-3/2} + C(\omega)n^{2\alpha - H - 3/2}\le C(\omega)n^{2\alpha - H - 3/2}.
\end{gather*}

Gathering all estimates we get
\begin{gather*}
\int_{t_N}^1 |D^\alpha_{t_N+}(\psi)(s)|ds
\leq C(\omega)\sum_{n=N}^\infty \Big(n^{2\alpha - H - 3/2}  + \sum_{j=N}^{n-1}j^{\alpha-H-3/2}\theta^{(\alpha-1)(j-n)}\Big)\\
\le C(\omega)\Big( N^{2\alpha - H-1/2} + \sum_{j=N}^
\infty j^{\alpha-H-3/2}\sum_{n=j+1}^\infty \theta^{(1-\alpha)(n-j)}  \Big)\\
\le C(\omega) N^{2\alpha - H-1/2},
\end{gather*}
which implies that $\wtinorm{\psi}_{\alpha,[t_N,T]}\to 0$, $N\to\infty$, finishing the proof.
\end{proof}

Now we turn to the main object of this article.\index{Wiener-transformable process}
 \begin{definition}\label{wti:def1} A Gaussian process  $G=\{G(t), t\in\wtiR^+\}$ is called $m$-Wiener-trans\-for\-mable if there exists $m$-dimensional Wiener process $W=\{W(t), t\in\wtiR^+\}$ such that $G$ and $W$ generate the same filtration, i.e. for any $t\in\wtiR^+$ $$\mathcal{F}_t^G=\mathcal{F}_t^W.$$
 We say that $G$ is $m$-Wiener-transformable to $W$ (evidently, process $W$ can be non-unique.)
 \end{definition}
 \begin{remark}\begin{itemize}
   \item[$(i)$] In the case when $m=1$  we say that the  process $G$ is  Wiener-trans\-formable.
   \item[$(ii)$] Being Gaussian so having moments of any order, $m$-Wiener-transformable process admits at each time $t\in\wtiR^+$
   the martingale    representation $G(t)=\wtiE(G(0))+\sum_{i=1}^m\int_0^tK_i(t,s)dW_i(s),$
   where $K_i(t,s)$ is $\mathcal{F}_s^W$-adapted for any $0\leq s\leq t$ and $\int_0^t\wtiE(K_i(t,s))^2ds<\infty$ for any $t\in\wtiR^+$.
   \end{itemize}
 \end{remark}

Now let the random variable $\xi$ be $\wtiF_T^W$-measurable, $\wtiE\xi^2<\infty$. Then in view of martingale representation theorem, $\xi$ can be represented as
\begin{equation}\label{wti:mart-repr}
\xi = \wtiE \xi + \int_0^T \vartheta(t) dW(t),
\end{equation}
where $\vartheta$ is an adapted process with $\int_0^T \wtiE\vartheta(t)^2 dt<\infty$.

As it was explained in introduction, we are interested when $\xi$ can be represented in the form
$$
\xi = \int_0^T \psi(s) dG(s),
$$
where the integrand is adapted, and the integral is understood in the pathwise sense.

\begin{theorem}\label{wti:thm1}
Let the following conditions hold.
\begin{itemize}
\item[$(i)$] Gaussian process $G$ satisfies condition $(A)$ and $(B)$.
\item[$(ii)$] Stochastic process $\vartheta$ in representation \eqref{wti:mart-repr} satisfies
\begin{equation}\label{wti:thetaassump}
\int_{0}^{T}|\vartheta(s)|^{2p}ds<\infty
\end{equation}
a.s.\ with some $p> 1$.
\end{itemize}
Then there exists a bounded adapted process $\psi$ such that $\left\|\psi\right\|_{\alpha,T}<\infty$ for some $\alpha\in \left(1-H,\frac{1}{2}\right)$ and $\xi$ admits the representation
 \begin{equation*}
\xi=\int_{0}^{T}\psi(s) dG(s),
\end{equation*}  almost surely.
\end{theorem}

\begin{remark} As it was mentioned in \cite{mish-shev}, it is sufficient to require the properties $(A)$ and $(B)$ to hold on some subinterval $[T-\delta,T]$. Similarly, it is enough to require in $(ii)$ that $\int_{T-\delta}^T |\vartheta(t)|^{2p}dt<\infty$ almost surely.
\end{remark}

First we prove a simple result establishing H\"older continuity of It\^o integral.

\begin{lemma}\label{wti:lem1}
Let $\vartheta=\{\vartheta(t), t\in [0,T]\}$ be a real-valued progressively measurable process such that for some $p\in(1,+\infty]$ $$\int_{0}^{T}|\vartheta(s)|^{2p}ds<\infty$$ a.s. Then the stochastic integral $\int_{0}^{t}\vartheta(s)dW(s)$ is H\"{o}lder continuous of any order up to $\frac{1}{2}-\frac{1}{2p}$.
\end{lemma}
\begin{proof}
First note that if there exist non-random positive constants $a,C$ such that for any $s,t\in [0,T]$ with $s<t$
$$
\int_{s}^{t}\vartheta^2(u)du \le C(t-s)^a,
$$
then $\int_0^t \vartheta(s)dW(s)$ is H\"older continuous of any order up to $a/2$. Indeed, in this case by the Burkholder inequality, for any $r>1$ and $s,t\in [0,T]$ with $s<t$
$$
\wtiE \left| \int_s^t \vartheta(u) dW(u)\right|^r\le C_r \wtiE \left( \int_s^t \vartheta^2(u) du\right)^{r/2} \le C (t-s)^{ar/2},
$$
so by the Kolmogorov--Chentsov theorem, $\int_0^t \vartheta(s)dW(s)$ is H\"older continuous of order $\frac{1}{r}(\frac{ar}2-1) = \frac{a}2 - \frac{1}{2r}$. Since $r$ can be arbitrarily large, we deduce the claim.

Now let for $n\ge 1$, $\vartheta_n(t) = \vartheta(t)\mathbf{1}_{\int_0^t |\vartheta(s)|^{2p} ds\le n}$, $t\in[0,T]$. By the H\"older inequality, for any $s,t\in [0,T]$ with $s<t$
$$
\int_s^t \vartheta_n^2(u)du \le (t-s)^{1-1/p}\left(\int_s^t |\vartheta(u)|^{2p} du\right)^{1/p} \le n^{1/p}(t-s)^{1-1/p}.
$$
Therefore, by the above claim, $\int_0^t \vartheta_n(s)dW(s)$ is a.s.\ H\"older continuous of any order up to $\frac12 - \frac{1}{2p}$. However, $\vartheta_n$ coincides with $\vartheta$ on $\Omega_n = \{\int_0^T |\vartheta(t)|^{2p} dt\le n\}$. Consequently, $\int_0^t \vartheta_n(s)dW(s)$ is a.s.\ H\"older continuous of any order up to $\frac12 - \frac{1}{2p}$ on $\Omega_n$. Since $\wtiProb(\bigcup_{n\ge 1} \Omega_n) = 1$, we arrive at the statement of the lemma.
\end{proof}

\noindent \textbf{Proof of Theorem~\ref{wti:thm1}.}
Define $$
Z(t) = \wtiE\xi + \int_0^t \vartheta(s) dW(s).
$$
This is an adapted process with $Z(T) = \xi$, moreover, it follows from Lemma~\ref{wti:lem1} that $Z$ is H\"older continuous of any order up to $\frac{1}{2}- \frac{1}{2p}$. Thus, the statement follows from Theorem~\ref{wti:thm:representation}.

In the case where one looks at improper representation, no assumptions on $\xi$ are needed. 
\begin{theorem}\label{wti:thm2} (Improper representation theorem)
Assume that an adapted Gaussian process $G=\{G(t)$, $t\in[0,T]\}$ satisfies conditions  $(A),(B)$. Then for any random variable $\xi$  there exists an adapted process $\psi$ that  $\left\|\psi\right\|_{\alpha,t}<\infty$ for some $\alpha\in \left(1-H,\frac{1}{2}\right)$ and any $t\in[0,T)$ and $\xi$ admits the representation
 \begin{equation*}
\xi=\lim_{t\to T-}\int_{0}^{t}\psi(s) dG(s),
\end{equation*}  almost surely.
\end{theorem}
\begin{proof}
The proof is exactly the same as for Theorem 4.2 in \cite{shev-viita}, so we just sketch the main idea. 

Consider an increasing sequence of points $\{t_n,n\ge 1\}$ in $[0,T)$ such that $t_n\to T$, $n\to\infty$, and let $\{\xi_n,n\ge 1\}$ be a sequence of random variables such that $\xi_n$ is $\mathcal{F}_{t_n}$-measurable for each $n\ge 1$, and $\xi_n\to \xi$, $n\to\infty$, a.s. Set for convenience $\xi_0 = 0$.
Similarly to Case I in Theorem~\ref{wti:thm:representation}, for each $n\ge 1$, there exists an adapted process $\{\phi_n(t),t\in[t_n,t_{n+1}]\}$, such that $\int_{t_n}^{t}\phi_n(s) d G(s)\to +\infty$ as $t\to t_{n+1}-$. For $n\ge 1$, define a stopping time
\begin{gather*}
\tau_n=\inf\left\{t\geq t_n: \int_{t_n}^{t}\phi_n(s) dG(s)\geq |\xi_n-\xi_{n-1}| \right\}
\end{gather*}
and set
\begin{gather*}
\psi(t)=\phi_n(t) \wtisign\big(\xi_n-\xi_{n-1}\big)\wtiindic_{[t_n,\tau_n]}(t), \,t\in[t_n,t_{n+1}).
\end{gather*}
Then for any $n\ge 1$, we have $\int_{0}^{t_{n+1}}\psi(s)dG(s)=\xi_n$ and $\int_{0}^{t}\psi(s)dG(s)$ lies between $\xi_{n-1}$ and $\xi_n$ for $t\in[t_{n-1},t_n]$. Consequently, $\int_{0}^{t}\psi(s) dG(s)\to \xi$, $t\to T-$, a.s., as required.
\end{proof}

Further we give several examples of Wiener-transformable Gaussian processes satisfying conditions $(A)$ and $(B)$  (for more detail and proofs see, e.g. \cite{mish-shev}) and formulate the corresponding representation results.

\subsection{Fractional Brownian motion}

Fractional Brownian motion\index{fractional Brownian motion} $B^H$ with Hurst parameter $H\in(0,1)$ is a centered Gaussian process with the covariance
$$
\wtiE B^H(t)B^H(s) = \frac{1}{2}\left(t^{2H}+s^{2H}-|t-s|^{2H}\right);
$$
an extensive treatment of fractional Brownian motion is given in \cite{Mish}.
For $H=\frac12$, fractional Brownian motion is a Wiener process; for $H\neq \frac12$ it is Wiener-transformable to the Wiener
process $W$ via relations
\begin{equation}\label{wti:fbmviawin}
B^H(t)=\int_0^t K^H(t,s) dW(s)
\end{equation}
and
\begin{equation}\label{wti:winviafbm}
 W(t)=\int_0^t k^H(t,s)dB^H(s),
\end{equation}
see e.g.\ \cite{norros}.

Fractional Brownian motion with index $H\in(0,1)$ satisfies condition $(A)$ and satisfies condition $(B)$ if $H\in(\frac{1}{2},1)$.

Therefore, a random variable satisfying \eqref{wti:thetaassump}  with any $p>1$ admits the representation \eqref{wti:reprez}.

\subsection{Fractional Ornstein--Uhlenbeck process}
Let $H\in (\frac{1}{2},1)$. Then the fractional Ornstein--Uhlenbeck process\index{fractional Ornstein--Uhlenbeck process} $Y=\{Y(t),   t\ge 0\}$, involving fractional Brownian component and  satisfying the equation
$$ Y(t)=Y_0+\int_0^t(b-aY(s))ds+\sigma  B^H(t),$$
where $a,b\in\wtiR$ and $\sigma>0$, is Wiener-transformable to the same Wiener process as the underlying fBm  $B^H$.

Consider a fractional Ornstein--Uhlenbeck process of the simplified form
\begin{equation*}
Y(t) = Y_0 + a \int_0^t Y(s) ds + B^H(t), \mbox{ } t \geq 0.
\end{equation*}
It satisfies condition $(A)$; if $a>0$, it satisfies condition $(B)$ as well.

As it was mentioned in \cite{mish-shev}, the representation theorem is valid for a fractional Ornstein-Uhlenbeck
process with a negative drift coefficient too. Indeed, we can annihilate the drift of the fractional Ornstein-Uhlenbeck process with the help of Girsanov theorem, transforming a fractional  Ornstein-Uhlenbeck process with negative drift to a fractional Brownian motion $\widetilde{B}^H$.
Then, assuming \eqref{wti:thetaassump}, we represent the random variable $\xi$ as  $\xi=\int_0^T\psi(s)d\widetilde{B}^H(s)$ on the new probability space. Finally, we return to the original probability space. Due to the pathwise nature of integral, its value is not changed upon changes of measure.


\subsection{Subfractional Brownian motion}
Subfractional Brownian motion\index{subfractional Brownian motion} with index $H$, that is a centered Gaussian process
$G^H=\left\{G^H(t), t \geq 0 \right\}$ with  covariance function
$$\wtiE G^H(t) G^H(s) = t^{2H}+s^{2H} -\frac{1}{2}\left(|t+s|^{2H} + |t-s|^{2H} \right),$$
 satisfies condition $(A)$ and  condition $(B)$ for $H\in(\frac{1}{2},1)$.

\subsection{Bifractional Brownian motion}
Bifractional Brownian motion\index{bifractional Brownian motion} with indices $A \in (0,1)$ and $K \in (0,1)$, that is a centered Gaussian process with covariance function
 $$ \wtiE G^{A,K}(t) G^{A,K}(s) = \frac{1}{2^K} \left( \left(t^{2A}+s^{2A}\right)^K - |t-s|^{2AK}\right),$$ satisfies condition $(A)$ with $H = AK$ and satisfies condition $(B)$ for $AK>\frac{1}{2}$.

\subsection{Geometric Brownian motion}\index{geometric Brownian motion}
Geometric Brownian motion involving the Wiener component  and having the form $$S=\left\{S(t)=S(0)\exp\left\{\mu t+\sigma W(t)\right\}, \;\; t\ge 0 \right\},$$
with $S(0)>0$, $\mu\in\wtiR$, $\sigma>0$, is Wiener-transformable to the underlying Wiener process $W$. However, it does not satisfy the assumptions of Theorem~\ref{wti:thm1}. One should appeal here to the standard semimartingale tools, like the martingale representation theorem.

\subsection{Linear combination of fractional Brownian motions}
Consider a collection of Hurst indices $\frac{1}{2}\le H_1< H_2<\ldots<H_m<1$ and independent fractional Brownian motions with corresponding Hurst indices $H_i$, $1\le i \le m$. Then the linear combination $\sum_{i=1}^{m}a_iB^{H_i}$ is $m$-Wiener-transformable to the Wiener process $W=(W_1,\ldots,W_m)$, where $W_i$ is such Wiener process to which fractional Brownian motion $B^{H_i}$ is Wiener-transformable. In particular,  the  mixed fractional Brownian motion  $M^H=W+B^H$, introduced in \cite{Cheridito}, is $2$-Wiener-transformable.

The linear combination $\sum_{i=1}^{m}a_iB^{H_i}$ satisfies condition  $(A)$   with $H=H_1$, and condition $(B)$ whenever $H_1>1/2$.

We note that in the case of mixed fractional Brownian motion, the existence of representation \eqref{wti:reprez} cannot be derived from Theorem~\ref{wti:thm1}, as we have $H = \frac12$ in this case. By slightly different methods, it was established in \cite{shev-viita} that arbitrary $\mathcal{F}_T$-measurable random variable $\xi$ admits the representation
$$
\xi = \int_0^T \psi(s) d\big(B^H(s) + W(s)\big),
$$
where the integral with respect to $B^H$ is understood, as here, in the pathwise sense, the integral with respect to $W$, in the extended It\^o sense. In contrast to Theorem \ref{wti:thm:representation}, we can not for the moment establish this result for the bounded strategies. Therefore, it would be interesting to study which random variables have representations with bounded $\psi$ in the mixed model.

\subsection{Volterra process}\index{Volterra process} Consider Volterra integral transform of Wiener process, that is the process of the form $G(t) = \int_0^t K(t,s) dW(s)$ with non-random kernel $K(t, \cdot) \in L_2[0,t]$ for $t\in[0,T]$. Let the constant $r\in[0,1/2)$ be fixed. Let the following conditions hold.
\begin{itemize}
\item[ $(B1)$] The kernel $K$ is non-negative on $[0,T]^2$  and for any $s\in [0,T]$   $K(\cdot,s)$ is non-decreasing in the first argument;
\item[$(B2)$]  There exist constants   $D_i>0, i=2,3$ and   $H\in(1/2,1)$    such that   $$|K(t_2,s) - K(t_1,s)| \leq D_2 |t_2-t_1|^{H}s^{-r},\quad s, t_1,t_2 \in [0,T] $$
      and  $$\ K(t,s)\leq D_3(t-s)^{H-1/2}s^{-r};$$
\end{itemize}
      and at least one of the following conditions
\begin{itemize}
    \item[$(B3,a)$]  There exist constant  $D_1>0$  such that $$
    D_1|t_2-t_1|^{H}s^{-r}\leq|K(t_2,s) - K(t_1,s)|,\quad s, t_1,t_2 \in [0,T];
    $$
     \item[$(B3,b)$]There exist constant  $D_1>0$  such that $$
     K(t,s)\geq D_1(t-s)^{H-1/2}s^{-r},\quad s, t \in [0,T].$$
\end{itemize}
Then the Gaussian  process  $G(t) = \int_0^t K(t,s) dW(s)$, satisfies  condition $(A)$, $(B)$  on any subinterval $[T-\delta, T]$ with $\delta\in (0,1)$.

\section{Expected utility maximization in Wiener-transformable markets}\label{wti:sec:4}

\subsection{Expected utility maximization for unrestricted capital profiles}\index{utility maximization}
Consider the problem of maximizing the expected utility. Our goal is to characterize the optimal asset profiles in the framework of the markets with  risky assets involving  Gaussian processes satisfying conditions of Theorem \ref{wti:thm1}. We follow the general approach described in \cite{ekel} and \cite{karat}, but apply its  interpretation from \cite{Foll-Sch}.  We fix $T>0$ and from now on consider $\wtiF_T^W$-measurable random variables. Let the utility function $u:\wtiR\rightarrow\wtiR$ be strictly increasing and strictly concave, $L^0(\Omega, \wtiF_T^W, \wtiProb)$ be the set of all $\wtiF_T^W$-measurable random variables, and let the set of admissible capital profiles coincide with $L^0(\Omega, \wtiF_T^W, \wtiProb)$. Let $\wtiProb^*$ be a probability measure on $(\Omega, \wtiF_T^W)$, which is equivalent to $\wtiProb$, and denote $\varphi(T)=\frac{d\wtiProb^*}{d\wtiProb}$.  The budget constraint is given  by $\wtiE_{\wtiProb^*}(X)=w$, where $w>0$ is some number that can be in some cases, but not obligatory, interpreted as the initial wealth. Thus the budget set is defined as
$$\mathcal{B}=\left\{X\in L^0\left(\Omega, \wtiF_T^{W}, \wtiProb\right)\cap L^1\left(\Omega, \wtiF_T^W, \wtiProb^* \right)|\wtiE_{\wtiProb^*}(X)=w\right\}.$$
The problem is to find such $X^*\in\mathcal{B}$, for which $\wtiE( u(X^*))=\max_{X\in\mathcal{B}} \wtiE( u(X))$. Consider the inverse function $I(x)=(u'(x))^{-1}$.
 \begin{theorem}[{\cite[Theorem 3.34]{Foll-Sch}}]\label{wti:Theorem main for max}
Let the following condition hold:
\label{wti:Follmer-Sch}
   Strictly increasing and strictly concave utility function $u:\wtiR\rightarrow\wtiR$ is continuously differentiable, bounded from above and $$\lim_{x\downarrow -\infty} u'(x)=+\infty.$$
Then the solution of this maximization problem has a form $$X^*=I(c\varphi(T)),$$   under additional assumption  that $\wtiE_{\wtiProb^*}(X^*)=w$.
\end{theorem}
To connect the solution of maximization problem with specific $W$-transform\-able Gaussian process describing the price process, we consider the following items.

1. Consider  random variable $\varphi(T)$,  $\varphi(T)>0$ a.s. and let $\wtiE(\varphi(T))=1.$  Being the terminal value of a positive martingale $\varphi=\{\varphi_t=\wtiE(\varphi(T)|\wtiF_t^W), t\in[0,T]\}$, $\varphi(T)$ admits the following representation
\begin{equation}\label{wti:fi1}
\varphi(T)=\exp\left\{\int_0^T \vartheta(s)dW_s-\frac{1}{2}\int_0^T \vartheta^2(s)ds\right\},
\end{equation}
  where $\vartheta$ is a real-valued progressively measurable process for which $$ \wtiProb\left\{\int_{0}^{T}\vartheta^2(s)ds<\infty\right\}=1.$$
Assume that  $\vartheta$ satisfies \eqref{wti:thetaassump}. Then $\varphi(T)$ is a terminal value of a H\"{o}lder continuous process of order $\frac{1}{2} - \frac{1}{2p}$.

2. Consider $W$-transformable Gaussian process $G=\{G(t), t\in[0,T]\}$ satisfying conditions $(A)$ and $(B)$, and introduce the set
\begin{gather*}
\mathcal{B}_w^G=\bigg\{\psi\colon [0,T]\times \Omega\to \wtiR\ \Big|\ \text{$\psi$ is bounded $\wtiF_t^W$-adapted, there exists a generalized}\\ \text{Lebesgue-Stieltjes integral}
 \int_0^T \psi(s)dG(s), \;\; \text{and} \;\; \wtiE\bigg(\varphi(T)\int_0^T\psi(s) dG(s)\bigg)=w\bigg\}.
\end{gather*}
\begin{theorem} \label{wti:TheoremIntRepresentation} Let the following conditions hold
\begin{itemize}
    \item[$(i)$] Gaussian process $G$ satisfies condition $(A)$ and $(B)$.
 \item[$(ii)$] Function $I(x),x\in\wtiR$ is H\"older continuous.
 \item[$(iii)$] Stochastic process $\vartheta$ in representation \eqref{wti:fi1} satisfies \eqref{wti:thetaassump} with some $p>1$.
        \item[$(iv)$] There exists $c\in\wtiR$ such that $\wtiE(\varphi(T)I(c\varphi(T)))=w$.
\end{itemize}  Then the random variable $X^*=I(c\varphi(T))$  admits the representation
\begin{equation}\label{wti:reprmain}
X^*=\int_0^T\overline{\psi}(s)dG(s),
\end{equation}
with some $\overline{\psi}\in \mathcal{B}_w^G,$
and
\begin{equation}\label{wti:maxim}\wtiE( u(X^*))=\max_{\psi\in\mathcal{B}_w^G}\wtiE \left(u\left(\int_0^T\psi(s)dG(s)\right)\right).\end{equation}
\end{theorem}

\begin{proof}
From Lemma \ref{wti:lem1} we have that for any $c \in \mathbb{R}$ the random variable $\xi= I(c \varphi(T))$  is the final value of a H\"{o}lder continuous process
$$
U(t)= I(c \varphi(t))  = I\left(c \exp\left\{\int_0^t \vartheta(s) d W(s) - \frac{1}{2} \int_0^t \vartheta^2(s) ds\right\}\right).
$$
and the H\"{o}lder exponent exceeds $\rho$. Together with $(i)$--$(iii)$ this allows to apply Theorem \ref{wti:thm1} to obtain the existence of representation (\ref{wti:reprmain}). Assume now that (\ref{wti:maxim}) is not valid, and there exists
$\psi_0 \in \mathcal{B}_w^G$ such that $\wtiE\left(\varphi(T)\int_0^T \psi_0(s) d G(s)\right)=w$, and $\wtiE u \left( \int_0^T \psi_0(s) d G(s) \right)>\wtiE u(X^*)$. But in this case $\int_0^T \psi_0(s) d G(s)$ belongs to $\mathcal{B}$, and we get a contradiction with Theorem \ref{wti:Theorem main for max}.
\end{proof}
\begin{remark}
Assuming only $(i)$ and $(iv)$, one can show in a similar way, but using Theorem~\ref{wti:thm2} instead of Theorem~\ref{wti:thm1} that
$$
\wtiE( u(X^*))=\sup_{\psi\in\mathcal{B}_w^G}\wtiE \left(u\left(\int_0^T\psi(s)dG(s)\right)\right).
$$
However, the existence of a maximizer is not guaranteed in this case.
\end{remark}

\begin{example}
Let $u(x) = 1 - e^{- \beta x}$ be an exponential utility function with constant absolute risk aversion $\beta>0$. In this case
$I(x) = - \frac{1}{\beta} \log ( \frac{x}{\beta})$. Assume that
$$\varphi(T) = \exp \left\{ \int_0^T \vartheta(s) dW(s) - \frac{1}{2} \int_0^T\vartheta^2(s) ds \right\}$$ is chosen in such a way that
\begin{equation}\begin{gathered}
\label{wti:eq1ex41}
\wtiE \left( \varphi(T) |\log \varphi(T)|\right)\\ =\wtiE \bigg( \exp\left\{ \int_0^T \vartheta(s) dW(s)  - \frac{1}{2} \int_0^T\vartheta^2(s) ds \right\}\\ \times
\left|\int_0^T \vartheta(s) dW(s) - \frac{1}{2} \int_0^T\vartheta^2(s) ds\right|
\bigg)<\infty.
\end{gathered}\end{equation}
Then, according to Example 3.35 from \cite{Foll-Sch}, the optimal profile can be written as
\begin{equation}
\label{wti:ex41optprofile}
X^* = - \frac{1}{\beta} \left( \int_0^T \vartheta(s) dW(s) - \frac{1}{2} \int_0^T\vartheta^2(s) ds \right) + w + \frac{1}{\beta} H(\wtiProb^*|\wtiProb),
\end{equation}
where $H(\wtiProb^*|\wtiProb) = \wtiE \left(\varphi(T) \log \varphi(T)\right)$, condition (\ref{wti:eq1ex41}) supplies that $H(\wtiProb^*|\wtiProb)$ exists, and the maximal value of the expected utility is
$$
\wtiE (u(X^*)) = 1 - \exp\left\{-\beta w - H(\wtiProb^*|\wtiProb) \right\}.
$$
Let $\varphi(T)$ be chosen in such a way that the corresponding process $\vartheta$ satisfies the assumption of Lemma \ref{wti:lem1}. Also, let $W$-transformable process $G$ satisfy conditions $(A)$ and $(B)$ of Theorem \ref{wti:Follmer-Sch}, and $\vartheta$ satisfy \eqref{wti:thetaassump} with $p>1$.
 Then  we can conclude directly from representation (\ref{wti:ex41optprofile}) that conditions of Theorem \ref{wti:Follmer-Sch} hold. Therefore, the optimal profile $X^*$  admits the representation  $X^* =   \int_0^T \psi (s) d G(s).$
\end{example}
 \begin{remark} Similarly, under the same conditions as above, we can conclude  that  for any constant $d\in \mathbb{R}$ there exists $\psi_d$ such that $X^* = d + \int_0^T \psi_d(s) d G(s).$ Therefore, we can start from any initial value of the capital and achieve the desirable wealth. In this sense, $w$ is not necessarily the initial wealth as it is often assumed in the semimartingale framework, but  is rather a budget constraint in the generalized sense.
 \end{remark}
 \begin{remark}
 In the case when $W$-transformable Gaussian process $G$ is a semimartingale, we can use Girsanov's theorem in order to get the representation, similar to \eqref{wti:reprmain}. Indeed,
  let, for example, $G$ be a Gaussian process of the form $G(t)=\int_0^t\mu(s)ds +\int_0^t a(s)dW(s)$, $|\mu(s)|\leq \mu$, $a(s)>a>0$ are non-random measurable functions, and $\xi$ is $\wtiF_T^W$-measurable random variable, $\wtiE(\xi^2)<\infty$. Then   we transform $G$ into $\widetilde{G}=\int_0^\cdot a(s)d\widetilde{W}(s)$, with the help of equivalent probability measure $\widetilde{\wtiProb}$ having Radon--Nikodym derivative $$\frac{d\widetilde{\wtiProb}}{d\wtiProb}=\exp\left\{-\int_0^T\frac{\mu(s)}{a(s)}dW(s)-\frac{1}{2}
\int_0^T\left(\frac{\mu(s)}{a(s)}\right)^2d s \right\}.$$ With respect to this measure $\wtiE_{\widetilde{\wtiProb}}|X^*|<\infty$, and we get the following  representation
\begin{gather}\label{wti:repraux}
X^*=\wtiE_{\widetilde{\wtiProb}}(X^*)+\int_0^T\psi(s)d\widetilde{W}_s=\wtiE_{\widetilde{\wtiProb}}(X^*)+\int_0^T
\frac{\psi(s)}{a(s)}d\widetilde{G}(s)\\=\wtiE_{\widetilde{\wtiProb}}(X^*)+\int_0^T
\frac{\psi(s)}{a(s)}d{G}(s)
=\wtiE_{\widetilde{\wtiProb}}(X^*)+\int_0^T\psi(s) {\mu(s)}  ds+\int_0^T
 {\psi(s)} dW(s).
\end{gather}
Representations \eqref{wti:reprmain} and \eqref{wti:repraux} have the following distinction: \eqref{wti:reprmain} ``starts'' from $0$ (but can start from any other constant) while \eqref{wti:repraux} ``starts'' exactly from $\wtiE_{\widetilde{\wtiProb}}(X^*)$.
\end{remark}

As we can see, the solution of the utility maximization problem for $W$--trans\-formable process depends on the process in  indirect way, through the random variable $\varphi(T)$ such that $\wtiE\varphi(T)=1$, $\varphi(T)>0$ a.s. Also, this solution depends on whether or not we can choose the appropriate value of $c$, but this is more or less a technical issue. Let us return to the choice of $\varphi(T)$. In the case of the semimartingale market, $\varphi(T)$ can be reasonably chosen as the likelihood ratio of some martingale measure, and the choice is unique in the case of the complete market. The non-semimartingale market can contain some hidden semimartingale structure. To illustrate this, consider two examples.

\begin{example}\label{wti:ex4.2}
Let the market consist of bond $B$ and stock $S$, $$B(t)=e^{rt},\;S(t)=\exp\left\{\mu t +\sigma B_t^{H}\right\},$$ $r\geq0$, $\mu\in\wtiR$, $\sigma>0$, $H>\frac{1}{2}$. The discounted price process has a form $Y(t)=\exp\left\{(\mu-r)t+\sigma B_t^H\right\}$. It is well-known that such market admits an arbitrage, but even in these circumstances the utility maximization problem makes sense.
Well, how to choose $\varphi(T)$? There are at least two natural approaches.

1. Note  that for $H>\frac{1}{2}$ the kernel $K^H$ from (\ref{wti:fbmviawin}) has a form
$$
K^H(t,s)= C(H) s^{\frac{1}{2}- H}\int_s^t u^{H-\frac{1}{2}}(u-s)^{H-\frac{3}{2}} du,
$$
and representation (\ref{wti:winviafbm}) has a form
$$
W(t) = \left(C(H)\right)^{-1}\int_0^t s^{\frac{1}{2}-H} K^*(t,s) d B_s^H,
$$
where
\begin{equation*}\begin{gathered}
K^*(t,s) = \Big(t^{H-\frac{1}{2}}(t-s)^{\frac{1}{2} - H} \\ -\left(H-\frac{1}{2}\right)\int_s^t u^{H-\frac{3}{2}}(u-s)^{\frac{1}{2}-H}du\Big)\frac{1}{\Gamma\left(\frac{3}{2} - H\right)}.
\end{gathered}\end{equation*}

Therefore,
\begin{eqnarray*}
& &\left(C(H)\right)^{-1}\int_0^t s^{\frac{1}{2}-H} K^*(t,s) d \left( (\mu - r) s + \sigma B_s^H\right)\\
&=& \sigma W(t) + \frac{\mu-r}{C(H)} \int_0^t s^{\frac{1}{2}-H} K^*(t,s) ds\\
&=& \sigma W(t) + \frac{\mu-r}{C(H)\Gamma\left(\frac{3}{2} - H\right)} \int_0^t\left( s^{\frac{1}{2}-H} t^{H- \frac{1}{2}} (t-s)^{\frac{1}{2}-H}\right.\\
& &{}- \left.\left(H-\frac{1}{2}\right)s^{\frac{1}{2}-H} \int_s^t u^{H-\frac{3}{2}}(u-s)^{\frac{1}{2}-H}du \right)ds\\
&=&\sigma W_t + \frac{\mu -r}{C(H) \Gamma(\frac{3}{2}-H)}\frac{\Gamma^2(\frac{3}{2}-H)}{(\frac{3}{2}-H)\Gamma (2 - 2 H)} t^{\frac{3}{2}-H}\\
&=& \sigma W_t + (\mu-r) C_1(H) t^{\frac{3}{2}-H},
\end{eqnarray*}
where
$$
C_1(H) = \left(\frac{3}{2}-H\right)^{-1} \left(\frac{\Gamma(\frac{3}{2} - H)}{2H\Gamma(2 - 2H)\Gamma(H+\frac{1}{2})} \right)^\frac{1}{2}.
$$
In this sense we say that the model involves a hidden semimartingale structure.\\
Consider a virtual semimartingale asset
\begin{align*}
\hat{Y}(t) &= \exp \left\{ (C(H))^{-1} \int_0^t s^{\frac{1}{2}-H}K^*(t,s) d \log Y(s) \right\} \\
&= \exp\left\{\sigma W_t + (\mu-r) C(H)t^{\frac{3}{2}-H}\right\}.
\end{align*}
We see that measure $\wtiProb^*$ such that
\begin{equation}\begin{split}\label{wti:ex42ChangeMeasure}
\frac{d\wtiProb^*}{d\wtiProb} &=\exp \left\{ - \int_0^T \left(\frac{(\mu-r)C_2(H)}{\sigma} s^{\frac{1}{2}-H}+ \frac{\sigma}{2}\right)dW_s\right. \\
&  \quad \left.- \frac{1}{2} \int_0^T \left(\frac{(\mu-r)C_2(H)}{\sigma} s^{\frac{1}{2}-H}+ \frac{\sigma}{2}\right)^2 ds\right\},
\end{split}\end{equation}
where $C_2(H) = C_1(H) \left(\frac{3}{2}-H\right),$ reduces $\hat{Y}(t)$ to the martingale of the form\linebreak $\exp \left\{\sigma W_t - \frac{\sigma^2}{2} t\right\}$.
Therefore, we can put $\varphi(T) = \frac{d\wtiProb^*}{d\wtiProb}$ from (\ref{wti:ex42ChangeMeasure}). Regarding the H\"{o}lder property, $\vartheta(s)= s^{\frac{1}{2}-H}$ satisfies \eqref{wti:thetaassump} with some $p>1$ for any $H\in(\frac12,1)$. Therefore, for the utility function $u(x) = 1 - e^{-\alpha x}$ we have
$$
X^* = \frac{1}{\alpha} \left(\int_0^T \varsigma(s) dW_s - \frac{1}{2} \int_0^T \varsigma_s^2 ds \right) + W + \frac{1}{2} H(\wtiProb^*|\wtiProb),
$$
where $\varsigma(s) = \frac{(\mu-r)C_2(H)}{\sigma} s^{\frac{1}{2}-H}+ \frac{\sigma}{2}$, and $|H(\wtiProb^*|\wtiProb)|<\infty.$

2. It was proved in \cite{Dung} that the fractional Brownian motion $B^H$ is the limit in $L_p(\Omega,\mathcal{F},\wtiProb)$ for any $p>0$ of the process
$$B^{H,\epsilon}(t)=\int_0^tK(s+\epsilon,s)dW(s)+\int_0^t\psi_\epsilon (s)ds,$$
where $W$ is he underlying Wiener process, i.e. $B^{H}(t)=\int_0^tK(t,s)dW(s),$ where
\begin{gather*}
K(t,s)=C_H s^{\frac{1}{2}-H}\int_s^t u^{H-\frac{1}{2}}(u-s)^{H-\frac{3}{2}}du,\\
\psi_\epsilon(s)=\int_0^s\partial_1K(s+\epsilon,u)dW_u,\\
\partial_1K(t,s)=\frac{\partial K(t,s)}{\partial t}=C_Hs^{\frac{1}{2}-H}t^{H-\frac{1}{2}}(t-s)^{H-\frac{3}{2}}.
\end{gather*}
Consider prelimit market with discounted risky asset price $Y^{\epsilon}$ of the form
$$Y^{\epsilon}(t)=\exp{\left\{(\mu-r)t+\sigma\int_{0}^{t}\psi_\epsilon(s)ds+\sigma\int_0^tK(s+\epsilon,s)dW_s\right\}}.$$
This financial market is arbitrage-free and complete, and the unique martingale measure has the Radon-Nikodym derivative
$$\varphi_\epsilon(T)=\exp\left\{-\int_0^T\zeta_\epsilon(t)dW_t-\frac{1}{2}\int_0^T\zeta^2_\epsilon(t)dt\right\},$$
where
$$\zeta_\epsilon(t)=\frac{\mu-r+\sigma\psi_\epsilon(t)}{\sigma K(t+\epsilon,t)}+\frac{1}{2}\sigma K(t+\epsilon,t).$$
Note that $K(t+\epsilon,t)\rightarrow 0$ as $\epsilon\rightarrow0$. Furthermore, $\rho_t=\frac{\mu-r+\sigma\psi_\epsilon(t)}{\sigma K(t+\epsilon,t)}$ is a Gaussian process with $\wtiE\rho_t=0$ and
\begin{gather*}
\wtivar \zeta_\varepsilon(t)=\int_{0}^{t}\left(\frac{\partial_1 K(t+\epsilon,u)}{ K(t+\epsilon,t)}\right)^2du\\
=\int_{0}^{t}\left(\frac{u^{1/2-H}(t+\epsilon)^{H-1/2}(t+\epsilon-u)^{H-3/2}}{t^{1/2-H}\int_t^{t+\epsilon}v^{H-1/2}(v-t)^{H-3/2}}\right)^2du\\
\ge\epsilon^{1-2H}\int_0^t(t+\epsilon-u)^{2H-3}du=\frac{\epsilon^{1-2H}t}{2-2H}\left(\epsilon^{2H-2}-(t+\epsilon)^{2H-2}\right)\rightarrow\infty.
\end{gather*}
Therefore, we can not get a reasonable limit of $\varphi_\epsilon(T)$ as $\epsilon\rightarrow0.$ Thus one should use this approach  with great caution.
\end{example}

\subsection{Expected utility maximization for restricted capital profiles}

Consider now the case when the utility function $u$ is defined on some interval $(a,\infty)$. Assume for technical simplicity that $a=0$. Therefore, in this case the set $\mathcal{B}_0$ of admissible capital profiles has a form
$$\mathcal{B}_0=\left\{X\in L^0(\Omega,\wtiF,\wtiProb):X\ge0 \;\; \text{a.s. and} \;\; \wtiE(\varphi(T) X)=w\right\}.$$
Assume that the utility function $u$ is continuously differentiable on $(0,\infty)$, introduce $\pi_1=\lim\limits_{x\uparrow\infty}u'(x)\ge0$, $\pi_2=u'(0+)=\lim\limits_{x\downarrow0}u'(x)\le +\infty$, and define $I^+:(\pi_1,\pi_2)\longrightarrow(0,\infty)$
as the continuous, bijective function, inverse to $u'$ on $(\pi_1,\pi_2)$.

Extend $I^+$ to the whole half-axis $\left[0,\infty\right]$ by setting
\begin{displaymath}
I^+(y)=\left\{ \begin{array}{ll}
+\infty,& y\le\pi_1\\
0,& y\ge\pi_2.
\end{array}\right.
\end{displaymath}

\begin{theorem}[\cite{Foll-Sch}, Theorem 3.39]
Let the random variable $X^*\in\mathcal{B}_0$ have a form $X^*=I^+(c\varphi(T))$ for such constant $c>0$ that $\wtiE( \varphi(T) I^{+}(c\varphi(T)))=w$. If $\wtiE u(X^*)<\infty$ then
$$\wtiE( u(X^*))=\max\limits_{X\in\mathcal{B}_0}\wtiE (u(X)),$$
and this maximizer is unique.
\end{theorem}

From here we deduce the corresponding result on the solution of utility maximization problem similarly to Theorem~\ref{wti:TheoremIntRepresentation}. Define, as before,
\begin{gather*}
\mathcal{B}_w^G=\bigg\{\psi\colon [0,T]\times \Omega\to \wtiR\ \Big|\ \text{$\psi$ is bounded $\wtiF_t^W$-adapted, there exists a generalized}\\ \text{Lebesgue-Stieltjes integral}
 \int_0^T \psi(s)dG(s)\ge 0, \;\; \text{and} \;\; \wtiE\bigg(\varphi(T)\int_0^T\psi(s) dG(s)\bigg)=w\bigg\}.
\end{gather*}
\begin{theorem}  Let the following conditions hold
\begin{itemize}
    \item[$(i)$] Gaussian process $G$ satisfies conditions $(A)$ and $(B)$.
 \item[$(ii)$] Function $I^+(x),x\in\wtiR$ is H\"older continuous.
 \item[$(iii)$] Stochastic process $\vartheta$ in representation \eqref{wti:fi1} satisfies \eqref{wti:thetaassump} with some $p> 1$.
        \item[$(iv)$] There exists $c\in\wtiR$ such that $\wtiE(\varphi(T)I^+(c\varphi(T)))=w$.
\end{itemize}  Then the random variable $X^*=I^+(c\varphi(T))$  admits the representation
\begin{equation*} 
X^*=\int_0^T\overline{\psi}(s)dG(s),
\end{equation*}
with some $\overline{\psi}\in \widetilde{\mathcal{B}}_w^G$. If $\wtiE u(X^*)<\infty$, the $X^*$ is the solution to expected utility maximization problem:
and
\begin{equation*}\wtiE( u(X^*))=\max_{\psi\in\widetilde{\mathcal{B}}_w^G}\wtiE \left(u\left(\int_0^T\psi(s)dG(s)\right)\right).\end{equation*}
\end{theorem}

\begin{example}
Consider the case of CARA utility function $u$. Let first $u(x)=\frac{x^\gamma}{\gamma}$, $x>0$, $\gamma\in(0,1)$. Then, according to \cite[Example 3.43]{Foll-Sch},
$$I^+(c\varphi(T))=c^{-\frac{1}{1-\gamma}}(\varphi(T))^{-\frac{1}{1-\gamma}}.$$
If $d:=\wtiE (\varphi(T))^{-\frac{\gamma}{1-\gamma}}<\infty$ then unique optimal profile is given by $X^*=\frac{w}{d}(\varphi(T))^{-\frac{1}{1-\gamma}}$, and the maximal value of the expected utility is equal to
$$\wtiE( u(X^*))=\frac{1}{\gamma}w^{\gamma}d^{1-\gamma}.$$
As it was mentioned,
\begin{equation}\label{wti:fi}
\varphi=\varphi(T)=\exp\left\{\int\limits_{0}^T\vartheta(s)dW(s)-\frac{1}{2}\int\limits_0^T \vartheta^2(s)ds,\right\}
\end{equation}
 thus $$(\varphi(T))^{-\frac{1}{1-\gamma}}=\exp\left\{-\frac{1}{1-\gamma}\int\limits_0^T \vartheta(s)dW(s)+\frac{1}{2(1-\gamma)}\int\limits_0^T \vartheta^2(s)ds\right\}.$$
Therefore, we get the following result.
\begin{theorem}
Let the process $\vartheta$ in the representation \eqref{wti:fi} satisfy \eqref{wti:thetaassump}, and
$$\wtiE\exp\left\{-\frac{\gamma}{1-\gamma}\int\limits_0^T\vartheta(s)dW_s+\frac{\gamma}{2(1-\gamma)}\int\limits_0^T\vartheta^2_sds\right\}<\infty.$$
Let the process $G$ satisfy the same conditions as in Theorem \ref{wti:TheoremIntRepresentation}. Then $X^*=\int\limits_0^T\psi(s)dG(s)$.
\end{theorem}
In the case where $u(x)=\log x$, we have $\gamma=0$ and $X^*=\frac{w}{\varphi(T)}$. Assuming that the relative entropy $H\left({\wtiProb}|{\wtiProb^*}\right)=\wtiE(\frac{1}{\varphi(T)}\log \varphi(T))$ is finite, we get that
$$\wtiE(\log X^*)=\log w + H\left({\wtiProb}|{\wtiProb^*}\right).$$
\end{example}

\section*{Conclusion}

We have studied a broad class of non-semimartingale financial market models, where the random drivers are Wiener-transformable Gaussian random processes, i.e. some adapted transformations of a Wiener process. Under assumptions that the incremental variance of the process satisfies two-sided power bounds, we have given sufficient conditions for random variables to admit integral representations with bounded adapted integrand; these representations are models for bounded replicating strategies. It turned out that these representation results can be applied to solve utility maximization problems in non-semimartingale market models.

\bigskip
\begin{petit}
\noindent
\textbf{Acknowledgements }
Elena Boguslavskaya is supported by Daphne Jackson fellowship funded by EPSRC.
The research of Yu.~Mishura was funded (partially) by the Australian Government through the Australian Research Council (project number DP150102758).
Yu.~Mishura acknowledges that the present research is carried through within the frame and support of the ToppForsk project nr. 274410 of the Research Council of Norway with title STORM: Stochastics for Time-Space Risk Models.
\end{petit}

\bibliographystyle{spmpsci}
\bibliography{wtbibl}
\end{document}